\documentclass[a4paper,11pt]{amsart}
\usepackage{}
\usepackage{amsfonts}
\usepackage{amssymb}
\usepackage{mathrsfs}
\usepackage{amsmath,amssymb,amsthm,latexsym,amscd,mathrsfs}
\usepackage{indentfirst}
 \newcommand{\ROM}[1]{\mathrm{\uppercase\expandafter{\romannumeral#1}}}
\numberwithin{equation}{section} \theoremstyle{plain}
\newtheorem{thm}{Theorem}[section]

\newtheorem{pro}{Proposition}[section]
\newtheorem{lem}{Lemma}[section]
\newtheorem{cor}{Corollary}[section]
\theoremstyle{definition}
\newtheorem{defn}{Definition}[section]

\newtheorem{exmp}{Example}[section]

\theoremstyle{remark}
\newtheorem{rem}{Remark}[section]

\newtheorem{ack}{Acknowledgements}   

\title[Finiteness theorems for equifocal hypersurfaces]
{Finiteness theorems for equifocal hypersurfaces}
\author[J. Q. Ge]{Jianquan Ge}
\address{School of Mathematical Sciences, Laboratory of Mathematics and Complex Systems, Beijing Normal
University, Beijing 100875, P.R. CHINA} \email{jqge@bnu.edu.cn}
\author[C. Qian]{Chao Qian}
\address{School of Mathematical Sciences, Laboratory of Mathematics and Complex Systems, Beijing Normal
University, Beijing 100875, P.R. CHINA}
\email{qianchao\_1986@163.com}
\author[Z. Z. Tang]{Zizhou Tang}
\address{School of Mathematical Sciences, Laboratory of Mathematics and Complex Systems, Beijing Normal
University, Beijing 100875, P.R. CHINA} \email{zztang@bnu.edu.cn}

\date{}

\thanks {The project is partially supported by the NSFC ( No.11071018 and No.11001016
), the SRFDP, and the Program for Changjiang Scholars and Innovative
Research Team in University.}

\begin{document}
\maketitle

\begin{abstract}
In this paper, we give a finiteness result on the diffeomorphism
types of curvature-adapted equifocal hypersurfaces in a simply connected compact
 symmetric space. Furthermore, the condition curvature-adapted can be dropped if the symmetric space is of rank one.
\end{abstract}

\section{Introduction}
\label{introduction}
A hypersurface $M^n$ in a real space form $N^{n+1}(c)$ with constant sectional curvature $c$ is said to be \emph{isoparametric} if
it has constant principal curvatures. Since the work of Cartan and M{\"u}nzner,
the subject of isoparametric hypersurfaces especially in the spherical case is rather fascinating to geometers. Hitherto, the classification problem has been almost completed except for one case
(see \cite{Th00} and \cite{Ce} for excellent surveys and \cite{CCJ07}, \cite{Imm08}, \cite{Chi23}, \cite{Miy09}, \cite{GX10} for recent progresses and applications).

In a general Riemannian manifold, a hypersurface is called \emph{isoparametric} if its nearby
parallel hypersurfaces have constant mean curvature. Note that this definition coincides with that in the case of real space forms above by a theorem of Cartan (cf. \cite{GTY}).  In particular, isoparametric hypersurfaces in a simply connected compact symmetric space have been found identical with equifocal hypersurfaces that introduced by Terng and Thorbergsson \cite{TT95}. In fact, they also introduced equifocal submanifolds of high codimensions and established similar structural results as the classical case of isoparametric hypersurfaces and submanifolds. It is worth mentioning that \cite{Tang98} obtained the possible values of the
multiplicities $(m_1,m_2)$ for equifocal hypersurfaces in rank two symmetric spaces , and \cite{Christ} generalized Thorbergsson's
result of the homogeneity of isoparametric submanifolds
in Euclidean spaces of codimension at least two (see \cite{Th91}) by showing that equifocal submanifolds in simply connected compact symmetric spaces of high codimensions must be homogeneous, and thus can be classified. However, the classification of equifocal hypersurfaces is still far from being reached.

In this paper, we endeavor to make some progress towards this classification problem by proving the finiteness of the diffeomorphism types of equifocal hypersurfaces. Our finiteness result relies on an additional condition that the equifocal hypersurfaces should be curvature-adapted. Nevertheless, our Theorem \ref{th1} generalizes the finiteness theorem for isoparametric hypersurfaces in spheres with four distinct principal curvatures given by Wu \cite{Wu94}, since every hypersurface in a sphere is curvature-adapted. Notice that our result also covers the finiteness conclusion for isoparametric hypersurfaces in spheres with six distinct principal curvatures which can not be derived by the method of Wu \cite{Wu94}. Now we state the finiteness theorem as the following.
\begin{thm}\label{th1}
Given a simply connected compact symmetric space $N$, there
are only finitely many diffeomorphism classes of curvature-adapted equifocal hypersurfaces in $N$.
\end{thm}

 Furthermore, the condition curvature-adapted above can be dropped if the symmetric space is of rank one, \textit{i.e.},
\begin{thm}\label{th3}
Given a simply connected compact rank one symmetric space $N$, there
are only finitely many diffeomorphism classes of equifocal hypersurfaces in $N$.
\end{thm}
In section \ref{in rank one}, we will give some examples of equifocal hypersurfaces in $\mathbb{C}P^n$ which are the images of one isoparametric hypersurface in $S^{2n+1}$ under the Hopf fibrations with different $S^1$-actions. It turns out that these equifocal hypersurfaces are of different diffeomorphism types, which illustrates the non-triviality of Theorem \ref{th3} since now one can not expect to prove this finiteness result directly from that in spheres by using the Hopf fibrations.

\section{Focal structure of equifocal hypersurfaces}\label{focal-point}
\subsection{Preliminaries}
In this subsection, we firstly recall some fundamental definitions and results
of \cite{TT95}.

Let $M$ be an immersed submanifold in a symmetric space $N$. The
normal bundle $\nu(M)$ is called: (i) \emph{abelian} if $exp(\nu(M)_x)$ is contained in some flat of $N$ for
each $x\in M$; (ii) \emph{globally flat} if the induced normal connection is flat and has
trivial holonomy. The end point map $\eta : \nu(M)
\rightarrow N$ is the restriction of the exponential map $exp$ to
$\nu(M)$. Let $v$ be a (local) normal vector field on $M$.
Then the end point map of $v$ is the map $\eta_v:M\rightarrow N$ defined
by $x\mapsto exp_x(v(x))$. If $v\in\nu(M)_x$ is a singular point of $\eta$ and the dimension of the kernel of $d\eta_v$ is $m$, then $v$ is called a \emph{multiplicity $m$ focal normal} and $exp(v)$ is called a \emph{multiplicity $m$ focal point} of $M$ with respect to $M$ in $N$. The \emph{focal data}, $\Gamma(M)$, is defined to be the set of all pairs $(v,m)$ such that $v$ is a multiplicity $m$ focal normal of $M$. The focal variety $\mathcal {V}(M)$ is the set of all pairs $(\eta(v),m)$ with $(v,m)\in\Gamma(M)$. If $v$ is a parallel normal field on $M$, then
$M_v:=\eta_v(M)$ is called the \emph{parallel set} defined by $v$.
The equifocal submanifolds are defined as the following.
\begin{defn}(\cite{TT95})
A connected, compact, immersed submanifold $M$ in a symmetric space
$N$ is called equifocal if
\begin{itemize}
\item[(1)] $\nu(m)$ is globally flat and abelian, and
\item[(2)] if $v$ is a parallel normal field on $M$ such that $\eta_v(x_0)$ is a
multiplicity $k$ focal point of M with respect to $x_0$, then
$\eta_v(x)$ is a multiplicity $k$ focal point of M with respect to
$x$ for all $x\in M$.
\end{itemize}
\end{defn}

Throughout this paper, we assume that $N=G/K$ is a simply connected compact symmetric
space, $\mathfrak{g}=\mathfrak{k}+\mathfrak{p}$ is
its Cartan decomposition, and $N$ is equipped with the $G$-invariant
metric $g$ given by the restriction of the negative of the Killing form $\langle~,~\rangle$
of $\mathfrak{g}$ to $\mathfrak{p}$. Then we restate a part of Theorem 1.6 of
\cite{TT95} in the following which will play a crucial role in the proof of our theorems later.
\begin{thm}(\cite{TT95}) \label{th0}
Let $M$ be an immersed, compact, equifocal hypersurface in the
simply connected compact symmetric space $N$, and $v$ a unit normal
vector field. Then the following hold:
\begin{itemize}
\item[(a)] Normal geodesics are circles of constant length, which will be
denoted by $l$.
\item[(b)] There exist integers $m_1$, $m_2$, an even number $2g$ and
$0<\theta<\frac{l}{2g}$ such that the focal points on the normal
circle $T_x:=exp(\nu(M)_x)$ are
$$x(j)=exp\Big(\Big(\theta+\frac{(j-1)l}{2g}\Big)v(x)\Big),\quad 1\leq j \leq2g,$$ and
their multiplicities are $m_1$ if $j$ is odd and $m_2$ if $j$ is
even.
\item[(c)] $M$ is embedded. Let $M_t:=\eta_{tv}(M)=\{exp(tv(x))|x\in M\}$
denote the set parallel to $M$ at distance $t$. Then $M_t$ is an
equifocal hypersurface and $\eta_{tv}$ maps $M$ diffeomorphically
onto $M_t$ if $t\in(-\frac{l}{2g}+\theta, \theta)$.
\item[(d)] $M_+:=M_{\theta}$ and $M_-:=M_{-\frac{l}{2g}+\theta}$ are embedded
submanifolds of codimension $m_1+1$, $m_2+1$ in $N$, and the maps
$\eta_{\theta v}:M\rightarrow M_+$ and $\eta_{(-\frac{l}{2g}+\theta)
v}:M\rightarrow M_-$ are $S^{m_1}$- and $S^{m_2}$-bundles respectively.
\item[(e)] The focal variety $\mathcal {V}(M)=(M_{+},m_1)\cup(M_{-},m_2)$.
\item[(f)] $\{M_t\mid t\in [-\frac{l}{2g}+\theta, \theta]\}$ gives a singular foliation of $N$, which is analogous to the orbit foliation of a cohomogeneity one isometric group action on $N$.
\item[(g)] $N=D_1\cup D_2$ and $D_1\cap D_2=M$, where $D_1$ and $D_2$ are
diffeomorphic to the normal disk bundles of $M_+$ and $M_-$
respectively.
\end{itemize}
\end{thm}

\subsection{Relation between focal points and shape operators}
In this subsection, we discuss the relation between focal points and shape operators
of submanifolds with abelian normal bundle in the simply connected compact symmetric space $N=G/K$.
In fact, this has been done in Section 3 of \cite{TT95}. For completeness, we repeat it as follows.

Let $G=Iso(N)$ and $M$ be a submanifold with abelian normal bundle in $N$.
Let $x_0 \in M$, $K=G_{x_0}$ and $\mathfrak{g}=\mathfrak{k}+\mathfrak{p}$ be the Cartan decomposition. Let $\pi:
G\rightarrow G/K=N$ be the canonical projection. For simplicity, we will denote by $\pi_*$ the restriction $\pi_*|_{\mathfrak{p}}$. Then
$\pi_*: \mathfrak{p}\rightarrow TN_{x_0}$ is an isomorphism, and
$TN_{x_0}$ is identified with $\mathfrak{p}$ by $\pi_*$. Moreover,
the curvature tensor of $N$ can be expressed as the following: $$R(\pi_*(X),
\pi_*(Y))\pi_*(Z)=\pi_*([[X,Y], Z]),\quad for~~ X, Y, Z\in \mathfrak{p}.$$

\begin{pro}\label{prop1}
With notations as above, we have $$R(v, TN_{x_0})v \subset
TM_{x_0},\quad for~~ v\in \nu(M)_{x_0}.$$
\end{pro}
\begin{proof}
For $w\in \nu(M)_{x_0}$,
\begin{eqnarray*}
g(R(v,TN_{x_0})v, w) &=&\langle[[\pi_*^{-1}v, \mathfrak{p}],
\pi_*^{-1}v], \pi_*^{-1}w
\rangle\\
&=& \langle[\pi_*^{-1}v, \mathfrak{p}],  [\pi_*^{-1}v,
\pi_*^{-1}w]\rangle\\
&=&0,
\end{eqnarray*}
where the last equality follows from the assumption that $\nu(M)$ is abelian.
\end{proof}

Next we recall an elementary fact concerning the tangential map $d\eta$ and Jacobi fields.
\begin{pro}\label{Jacobi}
Let $M$ be a submanifold of $N$, $x(s)$ a smooth curve in $M$,
$v(s)$ a normal field of $M$ along $x(s)$, and $\eta$
the end point map. Then $$d\eta_{v(0)}(v'(0))=J(1), $$ where $J$ is
the Jacobi field along $\gamma(t)=exp_{x(0)}(tv(0))$ satisfying the
initial condition $J(0)=x'(0)$ and
$J'(0)=-A_{v(0)}(x'(0))+\nabla^{\bot}_{x'(0)}v$, where $A_{v(0)}$
is the shape operator with respect to $v(0)$ and
$\nabla^{\bot}$ is the normal connection of the submanifold $M$.
\end{pro}

Given $v\in\nu(M)_{x_0}$, put $a=\pi_{*}^{-1}(v)\in \mathfrak{p}$.
Let $J(t)$ be a Jacobi field along the normal geodesic $\gamma(t)=exp_{x_0}(tv)$ in $N$. Denote
the parallel transport map along $\gamma$ from $\gamma(t_1)$ to
$\gamma(t_2)$ by $P_{\gamma}(t_1, t_2)$. Set
$Y(t)=\pi_{*}^{-1}P_{\gamma}(t, 0)J(t)$. Then the Jacobi equation
for $J$ gives rise to the following equation for $Y$:
\begin{equation}\label{eqn-Jacobi}
Y''-ad(a)^2Y=0.
\end{equation}

Let $\mathfrak{a}$ be a maximal abelian subalgebra in $\mathfrak{p}$
containing $a$ and $$\mathfrak{p}= \mathfrak{a}\oplus \sum_{\alpha
\in \Delta} \mathfrak{p}_\alpha$$ its root space decomposition, where
$ad(a)^2(z_{\alpha})=-\alpha(a)^2z_{\alpha}$ for $z_{\alpha}\in
\mathfrak{p}_\alpha$. By direct computation, we know that the
solution of the equation (\ref{eqn-Jacobi}) with the initial
conditions $Y(0)=p_0+\sum_{\alpha}p_{\alpha}$ and
$Y'(0)=q_0+\sum_{\alpha}q_{\alpha}$ is
$$Y(t)=p_0+tq_0+\sum_{\alpha}p_{\alpha}\cos{(\alpha(a)t)}+\sum_{\alpha}q_{\alpha}\frac{1}{\alpha(a)}\sin{(\alpha(a)t)},$$
where $p_0, q_0\in \mathfrak {a}$, $p_{\alpha}, q_{\alpha}\in
\mathfrak {p}_\alpha$ and $\lambda^{-1}\sin{\lambda}$ is defined to
be 1 if $\lambda=0$. For convenience, let $D_1(a)$ and $D_2(a)$ be
the operators defined as follows: for $z=p_0+\sum_{\alpha}p_\alpha\in\mathfrak {p}$ with $p_0\in \mathfrak {a}$ and
$p_{\alpha}\in \mathfrak {p}_\alpha$,
\begin{eqnarray*}
D_1(a)(z)&=&p_0 +\sum_{\alpha}p_{\alpha}\cos{(\alpha(a)t)},\\
D_2(a)(z)&=&p_0
+\sum_{\alpha}p_{\alpha}\frac{1}{\alpha(a)}\sin{(\alpha(a)t)}.
\end{eqnarray*}
Put
$R_a(z):=\pi_*^{-1}R(\pi_*(a),\pi_*(z))\pi_*(a)=-ad(a)^2(z)$. Then we have
$R_a\geq0$,
$D_1(a)=\cos(\sqrt{R_a})$ and
$D_2(a)=(\sqrt{R_a})^{-1}\sin(\sqrt{R_a})$, which imply that $D_1$
and $D_2$ depend only on $a$, but not on the choice of the maximal
abelian subalgebra $\mathfrak {a}$.
\begin{lem}\label{parallel} Let $x(s)$ be a curve in M,
$x(0)=x_0$ and $v(s)$ is a parallel normal field along $x(s)$ with
$v(0)=\pi_*a$ for some $a\in \mathfrak {p}$. Then
$$\pi_*^{-1}P_{\gamma}(1,0)d\eta_{v(0)}(v'(0))=\{D_1(a)-D_2(a)\pi_*^{-1}A_{v(0)}\pi_*\}(\pi_*^{-1}(x'(0))),$$
where $P_{\gamma}(1,0)$ is the parallel transport map along
$\gamma(t)=exp_{x_0}(tv(0))$ from $\gamma(1)$ to $\gamma(0)$.
\end{lem}
\begin{proof}
Let $V(s,t)=exp_{x(s)}(tv(s))$ be a variation of normal geodesics of
$M$, and $T=\frac{\partial V}{\partial s}$, $S=\frac{\partial
V}{\partial t}$. Then $S(s,0)=v(s)$, and $J(t)=T(0,t)$ is a Jacobi
field along the geodesic
$\gamma(t)=exp_{x(0)}(tv(0))=\pi(e^{ta})$ with $J(0)=x'(0)$. By
Proposition \ref{Jacobi}, we have $d\eta_{v(0)}(v'(0))=J(1)$ and
$J'(0)=-A_{v(0)}(x'(0))$ since $v(s)$ is parallel. As above, set
$Y(t)=\pi_{*}^{-1}P_{\gamma}(t, 0)J(t)$. Clearly,
$Y(0)=\pi_*^{-1}(x'(0))$ and
$Y'(0)=\pi_*^{-1}(J'(0))=-\pi_*^{-1}A_{v(0)}(x'(0))$. Then
$$\pi_*^{-1}P_{\gamma}(1,0)d\eta_{v(0)}(v'(0))=Y(1)=\{D_1(a)-D_2(a)\pi_*^{-1}A_{v(0)}\pi_*\}(\pi_*^{-1}(x'(0)))$$
by the definitions of $D_1(a)$ and $D_2(a)$.
\end{proof}

Now we can prove the following theorem of \cite{TT95} where the case of $N=G$ is explicitly presented and the general case has been abbreviated.
\begin{thm}\label{Th2.2}
Suppose M is a submanifold in N with abelian normal bundle and $a\in
\pi_*^{-1}\nu(M)_{x_0}\subset \mathfrak {p}$, $v=\pi_*(a)$. Then
\begin{itemize}
\item[(1)] the operator $(D_1(a)-D_2(a)\pi_*^{-1}A_{v}\pi_*)$ maps
$\pi_*^{-1}TM_{x_0}$ to itself.
\item[(2)] $\pi(e^a)$ is a focal point of $M$ of multiplicity $m$ with
respect to $x_0$ if and only if the operator
$(D_1(a)-D_2(a)\pi_*^{-1}A_{v}\pi_*)$ on $\pi_*^{-1}TM_{x_0}$
is singular with nullity $m$.
\end{itemize}
\end{thm}
\begin{proof}
Since part (1) follows straightforward from Proposition \ref{prop1}, it suffices to prove part (2). For the tangent space  $T\nu(M)_{(x_0,v)}$, we can choose
a natural basis which consists of vectors of the form $v'(0)$ as in
Lemma \ref{parallel} and $\sigma'(0)$ with
$\sigma(t)=exp_{x_0}(\pi_*(a+tb))=\eta(\pi_*(a+tb))$, $b\in
\pi_*^{-1}\nu(M)_{x_0}$. Since $\nu(M)$ is abelian, we have
$$d\eta(\sigma'(0))=\frac{d}{dt}|_{t=0}\pi(e^{a+tb})=\frac{d}{dt}|_{t=0}\pi(e^ae^{tb})=e^a_*(\pi_*(b))\neq0,$$
where $e^a_*$ denotes the tangential map of the $G$-action $e^a: N\rightarrow N$.
Now the theorem follows from Lemma \ref{parallel}.
\end{proof}

\subsection{Hypersurfaces in spheres}
In this subsection, to warm up we apply Theorem \ref{Th2.2} to investigate the focal structure of hypersurfaces in spheres. For simplicity, henceforth we identify $\mathfrak {p}$ with
$T_{x_0}N$ without referring to $\pi_*$.

Let $M^n$ be a hypersurface in the sphere $S^{n+1}$, $G=Iso(S^{n+1})=O(n+2)$. Note that the
normal bundle $\nu(M)$ is $1$-dimensional and thus abelian. Given $x_0 \in M,$ let
$K=G_{x_0}\cong O(n+1)$, $\mathfrak {g}=\mathfrak {k}+\mathfrak {p}$ be the Cartan
decomposition, and $a\in\mathfrak {p}$ be a unit vector normal to $M$ at
$x_0$. Then $a^{\bot}=TM_{x_0}$ is the tangent space of $M$ at $x_0$. Since $S^{n+1}$ has constant sectional curvature $1$, we have $R_a|_{a^\bot}=id$,
\begin{eqnarray}
&&D_1(ta)|_{a^\bot}=\cos(t\sqrt{R_a})|_{a^\bot}=\cos t~id,\nonumber\\
&&D_2(ta)|_{a^\bot}=(t\sqrt{R_a})^{-1}\sin(t\sqrt{R_a})|_{a^\bot}=\frac{\sin
t}{t}id.\nonumber
\end{eqnarray} Applying Theorem \ref{Th2.2}, we get the following proposition
immediately.
\begin{pro} \label{pro1}
With notations as above and for $t\in\mathbb{R}$,
$\exp_{x_0}(ta)$ is a focal point of $M$ in $S^{n+1}$ with respect to $x_0$ if and only
if $$\det(\cot t~ id-A_a)=0,$$
where $A_a$ is the shape operator of $M$ with respect to $a\in\nu(M)_{x_0}\subset\mathfrak {p}$.
\end{pro}
\begin{rem}\label{re1}
Denote by $L$ the number of focal points along the normal geodesic $\exp_{x_0}(ta)$, $t\in [0,\pi)$, of $M$ with respect to $x_0$. The proposition above implies
$$L=\sharp \{t\in [0,\pi)\mid\det(\cot
t~ id-A_a)=0\}\leq n,$$which is the Theorem 1 of \cite{Pin85}. Combining this with Theorem \ref{th0} shows that, for any given equifocal (isoparametric)
hypersurface $M$ in $S^{n+1}$, the distance between the two focal submanifolds satisfies $d(M_+,M_-)\geq\frac{\pi}{n+1}$, which says that
$d(M_+,M_-)$ has a lower bound that depends only on $S^{n+1}$. Such type
fact is crucial for our proof of the finiteness theorem later.
\end{rem}

\subsection{Hypersurfaces in simply connected compact symmetric spaces}
In this subsection, we will firstly apply Theorem \ref{Th2.2} to investigate the focal structure of hypersurfaces in the simply connected compact symmetric space $N$. From this focal structure we derive some corollaries when the hypersurface is curvature-adapted, or in addition it is equifocal.

Let $N^{n+1}$ be a simply connected compact symmetric space of dimension $n+1$ and rank $r$, $G=Iso(N)$, and $M^n$ be a hypersurface in $N$. Observe that
the normal bundle of a hypersurface is always abelian. Let $x_0\in
M$, $K=G_{x_0}$, $\mathfrak {g}=\mathfrak {k}+\mathfrak {p}$ be the Cartan
decomposition, and $a\in\mathfrak {p}$ be a unit vector normal to $M$ at
$x_0$. Let $\mathfrak {a}$ be a maximal abelian subalgebra in $\mathfrak {p}$
containing $a$, and $$\mathfrak {p}= \mathfrak {a}\oplus \sum_{\alpha
\in \Delta} \mathfrak {p}_\alpha$$ its root space decomposition. One can choose a basis for each $\mathfrak {p}_\alpha$ and $\mathfrak {a}$ so as to constitute a basis of $\mathfrak {p}=TN_{x_0}$ including $a$. Then since $R_a=-ad(a)^2$, under this basis we can diagonalize the operator $\sqrt{R_a}$ on $a^{\bot}=TM_{x_0}$ as
$$\sqrt{R_a}|_{a^\bot}=diag(\overbrace{d_1,...,
d_1}^{m_1},...,\overbrace{d_s,..., d_s}^{m_s},
\overbrace{0,...0}^{m_{s+1}})=diag(d_1I_{m_1},...,d_sI_{m_s}, 0_{m_{s+1}}),$$ where $d_i=|\alpha(a)|\geq0$ and $m_i=dim(\mathfrak {p}_\alpha)$ for some
$\alpha\in\Delta$, $i=1,...,s$, $s\leq n+1-r$ and $m_{s+1}=
r-1$. It is well known that the numbers $d_i$ usually depend on the choice of $x_0\in
M$ except for the case when $N$ has rank $r=1$. Set $\mathfrak {B}:=\sup\{|\alpha(b)|\mid\alpha\in\Delta,
b\in\mathfrak {a}, |b|=1\}$. Then $d_i\leq\mathfrak {B}$. Notice that $\mathfrak {B}$ is a finite positive number depending only on $N$, but not on the hypersurface $M$.

Applying Theorem \ref{Th2.2} will then derive the following
\begin{pro}\label{Pro24}
With notations as above and for $t\in\mathbb{R}$, $\exp_{x_0}(ta)$ is a focal point of $M^n$ in $N^{n+1}$ with respect to
$x_0$ if and only if
\begin{equation}\label{focal-formula}
\det(\left(
\begin{array}{cccc}
d_1\cot(td_1)I_{m_1}&&&\\

&\ddots&&\\
&&d_s\cot(td_s)I_{m_s}&\\
&&&\frac{1}{t}I_{m_{s+1}}
\end{array}
\right)-A_a)=0,
\end{equation}
where $A_a$ is the shape operator of $M$ with respect to $a$ and $d_i\cot(td_i)=\frac{1}{t}$ if $d_i=0$.
\end{pro}
\begin{proof}
It follows from the discussions above that under some basis of $TM_{x_0}$, we have
\begin{eqnarray}
&&\sqrt{R_a}|_{a^\bot}=diag(d_1I_{m_1},...,d_sI_{m_s}, 0_{m_{s+1}}),\nonumber\\
&&D_1(ta)|_{a^\bot}=\cos(t\sqrt{R_a})|_{a^\bot}=diag(\cos(td_1)I_{m_1},...,\cos(td_s)I_{m_s},I_{m_{s+1}}),\nonumber\\
&&D_2(ta)|_{a^\bot}=(t\sqrt{R_a})^{-1}\sin(t\sqrt{R_a})|_{a^\bot}=diag\Big(\frac{1}{td_1}\sin(td_1)I_{m_1},...,\frac{1}{td_s}\sin(td_s)I_{m_s},I_{m_{s+1}}\Big).\nonumber
\end{eqnarray}
Note that $A_{ta}=tA_{a}$ and thus
$$\det(D_1(ta)-D_2(ta)A_{ta})=\prod_{i=1}^s\frac{\sin(td_i)}{d_i}\det\Big(diag(d_1\cot(td_1)I_{m_1},...,d_s\cot(td_s)I_{m_s},\frac{1}{t}I_{m_{s+1}})-A_a\Big),$$
which immediately implies the conclusion by Theorem \ref{Th2.2}.
\end{proof}

Recall that a hypersurface $M$ is called \emph{curvature-adapted} if its shape operator $A_a$ commutes with the normal Jacobi operator $R_a|_{TM}$ for $a\in\nu(M)$, or equivalently, they are simultaneously diagonalizable. Then it follows from the proposition above the following corollary which can be regarded as a generalization of the theorem of Pinkall \cite{Pin85} in the spherical
case to more general ambient spaces (see Remark \ref{re1}).
\begin{cor} \label{c25}
With notations as above, suppose that $M^n$ is a curvature-adapted hypersurface and denote by $L$ the number of focal points along the normal geodesic $\exp_{x_0}(ta)$, $t\in [0,\frac{\pi}{\mathfrak{B}})$, of $M$ with respect to $x_0$, then
$$L=\sharp \{t\in [0,\frac{\pi}{\mathfrak
{B}})\mid\det\Big(diag(d_1\cot(td_1)I_{m_1},...,d_s\cot(td_s)I_{m_s},\frac{1}{t}I_{m_{s+1}})-A_a\Big)=0\}\leq (s+1)n,$$
where $s\leq n+1-r$.
\end{cor}
\begin{proof} The first equality follows immediately from Proposition \ref{Pro24}. Since $M$ is curvature-adapted, $A_a$ and
$\sqrt{R_a}|_{a^\bot}$ can be diagonalized simultaneously. Therefore,
the equality (\ref{focal-formula}) holds if and only if $d_i\cot(td_i)$ or $\frac{1}{t}$ equals some eigenvalue of $A_a$,
which can occur at most $(s+1)n$ times for $t\in [0,\frac{\pi}{\mathfrak
{B}})$. This proves the second inequality of the corollary.
\end{proof}

As a direct application we obtain the following estimate for a universal lower bound of the distance between the two focal submanifolds of any curvature-adapted equifocal hypersurface.
\begin{cor} \label{cor-lower-bound}
With notations as above, suppose that $M^n$ is a curvature-adapted equifocal hypersurface in $N$ and $M_{\pm}$ are the focal submanifolds defined in Theorem \ref{th0}, then the distance between the two focal submanifolds satisfies
\begin{equation}\label{lower-bound}
d(M_{+},M_{-})\geq\frac{\pi}{\mathfrak
{B}((n+2-r)n+1)},
\end{equation} i.e., $d(M_+,M_-)$ has a lower bound which only
depends on $N$.
\end{cor}
\begin{proof}
One can conclude from Theorem \ref{th0} that along a normal geodesic $\exp_{x_0}(ta)$, the focal points of $M$ in $N$ with respect to $x_0$ occur alternately and equidistantly in the two focal submanifolds $M_{+}$ and $M_{-}$. Therefore, the distance between any two succeeding focal points occurring in $\exp_{x_0}(ta)$, $t\in [0,\frac{\pi}{\mathfrak{B}})$, equals $d(M_{+},M_{-})$ and the distance between the last focal point and $\exp_{x_0}(\frac{\pi}{\mathfrak{B}}a)$ is no more than $d(M_{+},M_{-})$.
Hence we have $$(L+1)d(M_{+},M_{-})\geq \frac{\pi}{\mathfrak{B}},$$
 where $L$ is the number of focal points along the normal geodesic $\exp_{x_0}(ta)$, $t\in [0,\frac{\pi}{\mathfrak{B}})$, of $M$ with respect to $x_0$. Then applying the inequality of Corollary \ref{c25} will complete the proof.
\end{proof}

As another corollary of proposition \ref{Pro24}, we observe a direct proof of the following result which is a part of Theorem 1.4 of \cite{GTY} proved by some knowledge of algebraic geometry.
\begin{cor}\label{th2}(\cite{GTY})
A curvature-adapted equifocal hypersurface in a
simply connected compact rank one symmetric space has constant
principal curvatures.
\end{cor}
\begin{proof} Given a unit normal field $v$
for a curvature-adapted equifocal hypersurface $M^n$ in the simply connected compact symmetric space $N^{n+1}$, we have $n$ continuous functions, the principal curvatures
$\lambda_1\geq\lambda_2\geq...\geq\lambda_{n}$ on $M$. For $x\in M$ and $t\in\mathbb{R}$, we know from Proposition
\ref{Pro24} that, $\exp_{x}(tv)$ is a focal point of $M$ with respect to $x$ if
and only if
$$
\det(\left(
\begin{array}{cccc}
d_1\cot(td_1)I_{m_1}&&&\\

&\ddots&&\\
&&d_s\cot(td_s)I_{m_s}&\\
&&&\frac{1}{t}I_{m_{s+1}}
\end{array}
\right)-A_{v})=0,
$$
where $A_v$ is the shape operator of $M$ at $x$ with respect to $v$. Since now
$M$ is curvature-adapted, $A_v$ can be diagonalized simultaneously with $\sqrt{R_v}|_{v^\bot}$. Hence the equation above holds if and only if for some $t=t(x)\in\mathbb{R}$, $1\leq i\leq s$ and $1\leq k \leq n$,
$\lambda_k(x)= d_i\cot(td_i)$, or
$\lambda_k(x)=\frac{1}{t}$. On the other hand, since $M$ is equifocal, such functions $t=t(x)$ should be constant on $M$ by Theorem \ref{th0}. In fact, by Theorem \ref{th0},
each normal geodesic $exp_x(tv)$ is a circle of constant length $l$ and there exists an even number $2g$ and
$0<\theta<\frac{l}{2g}$, such that the focal points on each normal
circle $T_x=exp(\nu(M)_x)$ are
$x(j)=exp((\theta+\frac{(j-1)l}{2g})v(x))$, $1\leq j \leq2g$, which means that the functions $t(x)=\theta+\frac{(j-1)l}{2g}$ are constant on $M$. In conclusion, we have
$$Im\lambda_k\subset
\{d_i\cot{td_i}\mid t=\theta+\frac{j}{2g}l, j\in
\mathbb{Z}, 1\leq i\leq s\}\bigcup\{\frac{1}{t}\mid t=\theta+\frac{j}{2g}l, j\in
\mathbb{Z}\},\quad for~~ 1\leq k\leq n.$$
Now by the assumption that $N$ is a rank one symmetric space, we know that the numbers $d_i$, $i=1,...,s$, are constant on $M$ independent of the choice of $x\in M$. Finally,
since the principal curvature functions are continuous and the right set above is totally discontinuous, it follows that $\lambda_1,
\lambda_2,...,\lambda_{n}$ are constant on $M$.

The proof is now completed.
\end{proof}

\section{Proof of Theorem \ref{th1}}\label{proof}
In this section, based on the results in previous sections we are now able to prove the finiteness Theorem \ref{th1}.
Firstly we recall a general finiteness theorem for submanifolds proved in \cite{Cor90}.

Let $N$ be a Riemannian manifold, $f_1:M_1\rightarrow N$ and
$f_2:M_2\rightarrow N$ be two immersions of compact manifolds in
$N$. The immersions $f_1$ and $f_2$ are said to be \emph{equivalent} if there
exists a diffeomorphism $\varphi:M_1\rightarrow M_2$ and a homotopy
$F:M\times[0,1]\rightarrow N$ with $F_0=f_2\circ \varphi$, $F_1=f_1$,
and $F_t$ an immersion for any $t\in [0,1]$. For this notion of equivalent immersions, Corlette \cite{Cor90} proved the following finiteness result.
\begin{thm} (\cite{Cor90})\label{t32}
Let $N$ be a compact Riemannian manifold, $B$, $d$, and $v$ three
positive constants, and $\mathscr{I}$ the set of all immersions
$f:M\rightarrow N$ satisfying:
\begin{itemize}
\item[(1)] $M$ is a compact manifold;
\item[(2)] $|\Pi|\leq B$, where $|\Pi|$ is the pointwise operator norm
of the second fundamental form;
\item[(3)] either $diam(M)\leq d$ or $vol(M)\leq v$.
\end{itemize}
Then $\mathscr{I}$
contains only finitely many equivalence classes of immersions.
\end{thm}

Next we recall the following lemma in \cite{Gr04} on estimating the principal curvatures.

Let $M$ be a submanifold in a Riemannian manifold $N$. For any point
$x\in M$ and unit normal vector $v\in \nu(M)_x$, define $$\kappa(v):=sup\{\langle A_v
X,X\rangle \mid \|X\|=1,X\in TM_x\}=\emph{the maximal eigenvalue of $A_v$},$$ where $A_v$ denotes the
shape operator of $M$ with respect to $v$ at $x$. Also, we recall that the cut-focal radius of $M$ at $x$ in the direction $v$ is defined by  $$e_c(x,v):=sup\{t>0\mid d(exp_x(tv),M)=t\}.$$
\begin{lem}\label{L33}(cf. \cite{Gr04}, P.150, Lemma 8.9.)
 Let $M$ be a submanifold in a Riemannian manifold $N$ of nonnegative sectional curvatures. Then for any point
$x\in M$ and unit normal vector $v\in \nu(M)_x$, we have
$$\kappa(v)e_c(x,v)\leq 1.$$
\end{lem}

Now we are ready to prove the finiteness Theorem \ref{th1}.\\
\textbf{Proof of Theorem \ref{th1}.} Since any parallel hypersurface of a given curvature-adapted hypersurface in a symmetric space is also curvature-adapted, by Theorem \ref{th0}, without loss of generality we can assume that $M^n$ is a curvature-adapted equifocal hypersurface in the
simply connected compact symmetric space $N$ with $d(M, M_+)=d(M, M_-)$ and $v$ a given unit normal vector field. Moreover, $\{M_t=\eta_{tv}(M)\mid t\in(-d(M, M_+), d(M, M_+))\}$ is a family of parallel (diffeomorphic) curvature-adapted equifocal hypersurfaces that foliates the whole space $N$ with two singular varieties, the two focal submanifolds $M_{\pm}$. Meanwhile, it follows from Corollary \ref{cor-lower-bound} that there exists a positive number $D$ depending only on $N$ such that
$$d(M, M_+)=\frac{1}{2}d(M_+, M_-)\geq D.$$ Noticing that the volume $Vol(M_t)$ of $M_t$ depends continuously on $t\in(-d(M, M_+), d(M, M_+))$, we find that there exists some $\xi\in[-\frac{D}{2}, \frac{D}{2}]$ such that
$$Vol(N)=\int_{-d(M, M_+)}^{d(M,
M_+)}Vol(M_t)dt\geq\int_{-\frac{D}{2}}^{\frac{D}{2}}Vol(M_t)dt=D~Vol(M_\xi).$$ This shows
$Vol(M_\xi)\leq\frac{Vol(N)}{D}$. In addition, by the choice of $\xi$, we have
$$d(M_\xi,
M_+)\geq \frac{D}{2},\quad d(M_\xi, M_-)\geq \frac{D}{2}.$$

On the other hand, the structural results in Theorem \ref{th0} show that at any point $x$ in an equifocal hypersurface with respect to a unit normal vector $v$, the focal points coincide with the cut-focal
points and both are the points in the intersection of the normal geodesic circle $exp_x(\nu(M)_x)$ with the focal submanifolds $M_{\pm}$. Hence the cut-focal radius $e_c(x,v)$ is nothing but the distance from $x$ to $M_{+}$ or $M_{-}$ according to $exp_x(e_c(x,v)v)\in M_{+}$ or $exp_x(e_c(x,v)v)\in M_{-}$ respectively. In particular, by the discussion above, on the equifocal hypersurface $M_\xi$ we have for any point
$x\in M_{\xi}$ and unit normal vector $v\in \nu(M_{\xi})_x$, $$e_c(x,v)\geq\frac{D}{2}.$$
As it is well known, a compact simply connected symmetric space has nonnegative sectional curvatures. It follows from Lemma \ref{L33} that the maximal eigenvalue $\kappa(v)$ of the shape operator $A_v$ of $M_\xi$ satisfies
$$\kappa(v)\leq\frac{1}{e_c(x,v)}\leq \frac{2}{D}.$$
Since this inequality holds for any unit normal vector $v\in \nu(M_{\xi})_x$, it follows that each eigenvalue $\lambda$ of $A_v$ satisfies $$|\lambda|\leq \frac{2}{D}.$$ This shows that the operator norm $|\Pi|$ of the second fundamental form at $x\in M_\xi$ satisfies
$$|\Pi|\leq \frac{2}{D}.$$
Hence $M_{\xi}$ satisfies all the conditions of Theorem \ref{t32} so that there are only finitely many equivalence classes of such immersions of curvature-adapted equifocal hypersurfaces.

The proof is now completed.
$\hfill{\square}$
\section{Equifocal hypersurfaces in compact rank one symmetric spaces}\label{in rank one}
In this section, we give a more detailed study for equifocal hypersurface in compact rank one symmetric spaces.

First, we investigate some examples of equifocal hypersurfaces in complex projective spaces.
We will construct them through Hopf fibrations by projecting the OT-FKM-type isoparametric hypersurfaces in spheres which almost cover all isoparametric hypersurfaces with four distinct principal curvatures in spheres (cf. \cite{CCJ07}, \cite{Chi23}, \cite{Imm08}). Now we recall some fundamental definitions. For a symmetric Clifford system ${A_0,...,A_m}$ on $\mathbb{R}^{2l}$, \textit{i.e.}, $A_i$'s are symmetric matrices satisfying $A_iA_j+A_jA_i=2\delta _{ij}I_{2l}$, the OT-FKM-type isoparametric polynomial $F$ on $\mathbb{R}^{2l}$ is then defined as (cf. \cite{FKM81}):
$$F(z)=|z|^4-2\sum_{p=0}^{m}\langle A_pz,z\rangle^2,$$
where we take the coordinate system $z=(x^t,y^t)^t=(x_1,...,x_l,y_1,...,y_l)^t\in\mathbb{R}^{2l}$. By orthogonal transformations, without loss of generality we can write
$$A_0=\left(
\begin{array}{cc}
I&0\\
0&-I
\end{array}
\right),\quad
\ A_1=\left(
\begin{array}{cc}
0&I\\
I&0
\end{array}
\right),$$
$$A_j=\left(
\begin{array}{cc}
0&-E_j\\
E_j&0
\end{array}
\right),\quad j=2,...,m,$$
where $\{E_2,...,E_m\}$ is a skew-symmetric Clifford system on $\mathbb{R}^l$, \textit{i.e.}, $E_i$'s are skew-symmetric matrices satisfying $E_iE_j+E_jE_i=-2\delta _{ij}I_{l}$. It can be verified that the level hypersurfaces of this polynomial restricted to the unit sphere have 4 distinct constant principal curvatures with multiplicities $m_1=m$ and $m_2=l-m-1$, provided $l-m-1> 0$. By using the well-known Hopf fibration $\pi:S^{2n+1}\rightarrow\mathbb{C}P^{n}$, \cite{Wan82} proved that a hypersurface $M$ in $\mathbb{C}P^{n}$ is isoparametric (equifocal) if and only if its inverse image $\pi^{-1}(M)$ under the Hopf fibration $\pi$ is an isoparametric hypersurface in $S^{2n+1}$. However, given one isoparametric hypersurface in $S^{2n+1}$, it can be neither projectable nor projected uniquely up to isometry through the Hopf fibrations with different $S^1$-actions. What is more, the induced equifocal hypersurfaces in $\mathbb{C}P^{n}$ may have different diffeomorphism types as the following examples will show.
\begin{exmp}\label{example41}
Consider the isoparametric hypersurface in spheres of OT-FKM-type with $m=1$, $l\geq 4$ .
 For $z=(x^t,y^t)^t\in \mathbb{R}^{2l}=\mathbb{R}^{l}\oplus \mathbb{R}^{l}$, $$F(z)=|z|^4-2\{\langle A_0z,z\rangle^2+\langle A_1z,z\rangle^2\},$$
 then
  \begin{eqnarray*}
&&M^{2l-2}=F^{-1}(0)\cap S^{2l-1},\\
&&M_+^{2l-3}= F^{-1}(1)\cap S^{2l-1}=\{(x^t,y^t)^t\mid|x|^2=|y|^2=\frac{1}{2}, \langle x,y\rangle=0\},\\
&&M_-^{l}=F^{-1}(-1)\cap S^{2l-1}=\{(x^t,y^t)^t\mid|x|^2+|y|^2=1, |x|^2|y|^2=\langle x,y\rangle^2\},
\end{eqnarray*}
are the isoparametric hypersurface and the two focal submanifolds in $S^{2l-1}$ respectively. Define $$\Phi: S^1\times S^{l-1}\rightarrow M_-,\ (e^{i\varphi}, w)\mapsto  (\cos{\varphi}\cdot w,~ \sin{\varphi}\cdot w),$$  then we observe that $\Phi$ is a two-to-one covering map. Additionally, $\Phi: S^1\times S^{l-1}\rightarrow S^{2l-1}$ is an isometric immersion from the standard product $S^1\times S^{l-1}$. Hence, we get the isometric diffeomorphism: $M_-\cong (S^1\times S^{l-1})/\mathbb{Z}_2$.

1. Define a complex structure $J:\mathbb{R}^{2l}\rightarrow\mathbb{R}^{2l}$ by $(x^t,y^t)^t\mapsto (y^t,-x^t)^t$. And the corresponding $S^1$-action on $\mathbb{R}^{2l}$ is defined as: $e^{i\theta}\cdot z=\cos{\theta}z+\sin{\theta}J(z)$. Clearly $F$ is $S^1$ invariant, \textit{i.e.}, $F(e^{i\theta}\cdot z)=F(z)$ for any $z\in \mathbb{R}^{2l}$ and $e^{i\theta}\in S^1$. Denote by $\pi_J:S^{2l-1}\rightarrow\mathbb{C}P^{l-1}$ the associated Hopf fibration. Hence, by \cite{Wan82}, $\widetilde{M}^{2l-3}=M^{2l-2}/S^1=\pi_J(M^{2l-2})$ is the isoparametric hypersurface in $\mathbb{C}P^{l-1}$ corresponding to $M^{2l-2}$ in $S^{2l-1}$, and $\widetilde{M}_+^{2l-4}=M_+^{2l-3}/S^1=\pi_J(M_+^{2l-3})$, $\widetilde{M}_-^{l-1}=M_-^{l}/S^1=\pi_J(M_-^{l})$ are the corresponding focal submanifolds in $\mathbb{C}P^{l-1}$ respectively. As defined and calculated in \cite{GTY}, the $\alpha$-invariant is constant on each level hypersurface of $F|_{S^{2l-1}}$ in this case, which implies that $\widetilde{M}^{2l-3}$ is homogeneous in $\mathbb{C}P^{l-1}$. In order to identify $\widetilde{M}_-^{l-1}$, we need to determine how $S^1$ acts on $M_-^{l}$. Since $M_-^{l}=\Phi(S^1\times S^{l-1})$ as observed above, one can see that $S^1$ acts on $M_-$ as
  \begin{eqnarray*}
e^{i\theta}\cdot\Phi(e^{i\varphi},w) &=&e^{i\theta}\cdot(\cos{\varphi}w, \sin{\varphi}w)\\
&=& \cos{\theta}(\cos{\varphi}w, \sin{\varphi}w)+\sin{\theta}(\sin{\varphi}w,-\cos{\varphi}w)\\
&=&(\cos{(-\theta+\varphi)}w, \sin{(-\theta+\varphi)}w)\\
&=&\Phi(e^{i(-\theta+\varphi)},w).
\end{eqnarray*}
Consequently we have a diffeomorphism: $\widetilde{M}_-^{l-1}\cong((S^1\times S^{l-1})/\mathbb{Z}_2)/S^1\cong \mathbb{R}P^{l-1}$. This implies that $\pi_1(\widetilde{M}^{2l-3})=\pi_1(\widetilde{M}_-^{l-1})=\mathbb{Z}_2$ since $\widetilde{M}^{2l-3}$ is an $S^{l-2}$-bundle over $\widetilde{M}_-^{l-1}$ and $l\geq 4$ as assumed.

2. Assume $l$ is even. Denote $l=2n+2,\ n\geq1$. We want to define another complex structure $J^{'}$. For $z=(x^t,y^t)^t\in \mathbb{R}^{4n+4}=\mathbb{R}^{2n+2}\oplus \mathbb{R}^{2n+2}$, define $J^{'}z:=((Tx)^t, (Ty)^t)^t$ where $T:\mathbb{R}^{2n+2}\rightarrow\mathbb{R}^{2n+2}$ is a linear transformation satisfying $T^2=-I,\ T^t=-T$. Then we have another $S^1$-action on $\mathbb{R}^{4n+4}$ with respect to $J^{'}$ under which $F$ is also invariant. Denote by $\pi_{J^{'}}:S^{4n+3}\rightarrow\mathbb{C}P^{2n+1}$ the associated Hopf fibration, and let $\widetilde{M^{'}}^{4n+1}$, $\widetilde{M^{'}_+}^{4n}$, and $\widetilde{M^{'}_-}^{2n+1}$ be the corresponding isoparametric hypersurface and the two focal submanifolds in $\mathbb{C}P^{2n+1}$ respectively. By a direct computation, the invariant $\Omega _F$ defined in \cite{GTY} is
 \begin{eqnarray*}
\Omega _F(z) &:=&DF^t\cdot J^{'}\cdot D^2F\cdot J^{'}\cdot DF|_{S^{4n+3}}(z)\\
&=& 64\{2F^2(z)-F(z)-2+16(\langle A_0z,z\rangle^2+\langle A_1z,z\rangle^2)\langle A_0z,J^{'}A_1z\rangle^2\}\\
&=& 64\{2F^2(z)-F(z)-2+64(\langle A_0z,z\rangle^2+\langle A_1z,z\rangle^2)\langle x,Ty\rangle^2\}.
\end{eqnarray*}
For convenience, we take $T(x_1,...x_{2n+2})^t=(x_{n+2}, ..., x_{2n+2}, -x_1,...-x_{n+1})^t$ for $x\in \mathbb{R}^{2n+2}$.
In $M^{4n+2}=F^{-1}(0)\cap S^{4n+3}$, choose two points $z=(x^t,y^t)^t$ and $\hat{z}=(\hat{x}^t,\hat{y}^t)^t$ with $x_i=\sqrt{\frac{1}{2}+\frac{1}{2\sqrt{2}}}\delta_{i1}$, $y_i=\sqrt{\frac{1}{2}-\frac{1}{2\sqrt{2}}}\delta_{i~n+2}$, and $\hat{x}_i=\sqrt{\frac{1}{2}+\frac{1}{2\sqrt{2}}}\delta_{i1}$, $\hat{y}_i=\sqrt{\frac{1}{2}-\frac{1}{2\sqrt{2}}}\delta_{i2}$, for $1 \leq i \leq 2n+2$. Then $\Omega _F(z)=128$ and $\Omega _F(\hat{z})=-128$, \textit{i.e.}, $\Omega _F$ is not constant on $M^{4n+2}$. Then we get that the $\alpha$-invariant defined in \cite{GTY},
$$\alpha=\frac{1}{g^3(1-F^2)^{\frac{3}{2}}}\{g^3F(3-2F^2)+\Omega _F\},$$ is not constant on $M^{4n+2}$, which implies that $\widetilde{M^{'}}^{4n+1}$ is not homogeneous in $\mathbb{C}P^{2n+1}$. Now we are going to identify $\widetilde{M^{'}_-}^{2n+1}$. For this, we also need to determine how $S^1$ acts on $M_-^{2n+2}$ in this case. In fact, in this case $S^1$ acts on $M_-^{2n+2}=\Phi(S^1\times S^{2n+1})$ as
  \begin{eqnarray*}
e^{i\theta}\cdot\Phi(e^{i\varphi},w) &=&e^{i\theta}\cdot(\cos{\varphi}w, \sin{\varphi}w)\\
&=& \cos{\theta}(\cos{\varphi}w, \sin{\varphi}w)+\sin{\theta}(\cos{\varphi}Tw, \sin{\varphi}Tw)\\
&=&(\cos{\varphi}(\cos{\theta}w+\sin{\theta}Tw),\  \sin{\varphi}(\cos{\theta}w+\sin{\theta}Tw))\\
&=&\Phi(e^{i\varphi},\cos{\theta}w+\sin{\theta}Tw).
\end{eqnarray*}
Consequently we have a diffeomorphism $\widetilde{M^{'}_-}^{2n+1}\cong S^1\times\mathbb{C}P^{n}$. This implies that $\pi_1(\widetilde{M^{'}}^{4n+1})=\pi_1(\widetilde{M^{'}_-}^{2n+1})=\mathbb{Z}$ since $\widetilde{M^{'}}^{4n+1}$ is an $S^{2n}$-bundle over $\widetilde{M^{'}_-}^{2n+1}$.
\end{exmp}
This example shows the following
\begin{cor}
For any $n\geq 1$, there exists an isoparametric hypersurface $M^{4n+2}$ in $S^{4n+3}$, from which we get two non-congruent $S^1$-quotient equifocal hypersurfaces in $\mathbb{C}P^{2n+1}$ by choosing different complex structures $J$ and $J^{'}$ on $\mathbb{R}^{4n+4}$. Moreover, these two equifocal hypersurfaces in $\mathbb{C}P^{2n+1}$ are not homotopy equivalent and thus have different diffeomorphism types.
\end{cor}
\begin{rem}Even if the isoparametric hypersurfaces in spheres are classified completely, we can not get the classification for isoparametric hypersurfaces in $\mathbb{C}P^{n}$  or in $\mathbb{H}P^{n}$ directly by using the Hopf fibrations. We should be careful that different complex structures may induce different $S^1$-actions, and different $S^1$-actions may give non-diffeomorphic quotient submanifolds in $\mathbb{C}P^{n}$  or in $\mathbb{H}P^{n}$. Thereby Theorem \ref{th3} of us really makes sense in this viewpoint.
\end{rem}

To prepare for the proof of Theorem \ref{th3} we need the following remarkable equality established by Thorbergsson involving $g$, $ m_1$, $m_2$ for an equifocal hypersurfce $M$ in a simply connected symmetric space $N$.
\begin{pro}(\cite{Th})
Let $i$ denote the index of $\gamma |_{[0,2\pi]}$ as a critical point of the energy functional $E$ in the path space $\Omega _{pp}$, where $\gamma |_{[0,2\pi]}$ is a closed geodesic normal to $M$ and $\gamma$ is parameterized such that its minimal period is $2\pi$,  and let $v$ denote its nullity. Then we have $$g(m_1+m_2)=i+v.$$
\end{pro}
\begin{rem}\label{rem42}
For $N=S^n,\ \mathbb{C}P^n,\ \mathbb{H}P^n$ and the Cayley projective plane $CaP^2$, the equality of Thorbergsson will give the well-known formulas: $g(m_1+m_2)=2(n-1)$ for $S^n$, $g(m_1+m_2)=2n$ for $\mathbb{C}P^n$, $g(m_1+m_2)=4n+2$ for $\mathbb{H}P^n$ and $g(m_1+m_2)=22$ for $CaP^2$.
\end{rem}

Now we are ready to prove Theorem \ref{th3}.\\
\textbf{Proof of Theorem \ref{th3}.}
Let $M$ be an equifocal hypersurface in a simply connected compact rank one symmetric space $N$, $M_+$ and $M_-$ be the focal submanifolds. One can conclude from Theorem \ref{th0} that along a normal geodesic $\exp_{x}(tv(x))$, the focal points of $M$ in $N$ with respect to $x\in M$ occur alternately and equidistantly in the two focal submanifolds $M_{+}$ and $M_{-}$ and normal geodesics are closed. Note that a simply connected compact symmetric space has rank one if and only if all its
geodesics are closed, and we will assume that the Riemannian metrics on these spaces are normalized such that their closed geodesics are
of length $2\pi$. By Remark \ref{rem42} and $2gd(M_+, M_-)=2\pi$, we can get a lower bound on $d(M_+,M_-)$ depends only on $N$. In fact, we have $d(M_+,M_-)= \pi/g\geq \pi/n$ for $\mathbb{C}P^n$, $d(M_+,M_-)= \pi/g\geq \pi/(2n+1)$ for $\mathbb{H}P^n$ and $d(M_+,M_-)= \pi/g\geq \pi/11$ for $CaP^2$. Then mimicking the proof of Theorem \ref{th1}, we finally complete the proof of Theorem \ref{th3}. \hfill $\Box$
\begin{ack}
The authors would like to thank Professor Xiaochun Rong and Professor Gudlaugur Thorbergsson for their valuable suggestions and helpful comments
during the preparation of this paper.
\end{ack}


\end{document}